\documentclass[leqno]{amsart}
\usepackage{amsmath,amssymb}
\usepackage{a4,amsfonts}
 \newtheorem{thm}{Theorem}[section]

 \newtheorem{lemma}[thm]{Lemma}
 \newtheorem{prop}[thm]{Proposition}
 \newtheorem{hy}[thm]{Hypotheses}
 \newtheorem{defin}[thm]{Definition}
 \newtheorem{rem}[thm]{Remark}

\newtheorem{hy1}[thm]{Hypothesis}

\newtheorem{ex}[thm]{Examples}
\newtheorem{ex1}[thm]{Example}
 
 \DeclareMathOperator{\IM}{Im}
 \DeclareMathOperator{\Ker}{Ker}

\newcommand{\s}{\hspace{4ex}}

\newcommand{\ep}{\varepsilon}

\newcommand{\acc}{\"}

\newcommand{\finedim}{{\unskip\nobreak\hfil\penalty50
   \hskip2em\hbox{}\nobreak\hfil\mbox{\rule{1ex}{1ex} \qquad}
   \parfillskip=0pt \finalhyphendemerits=0\par\medskip}}

\title[Hierarchies of Evolution Equations]{The Sturm-Liouville Hierarchy of evolution equations and the Weyl functions of  related spectral problems}

\author{Paola Rubbioni, Anna Rita Sambucini, Luca Zampogni}

\address{Dipartimento di Matematica e Informatica, Università di Perugia}
\email{paola.rubbioni@unipg.it, anna.sambucini@unipg.it, luca.zampogni[unipg.it}

\begin{document}
\begin{abstract} 
The goal of the present paper  is that of defining the so-called  Sturm-Liouville hierarchy of evolution equations, firstly by using the zero-curvature formalism, then by using the asymptotic properties of the Weyl $m$-functions for certain classes of pairs $(q,y)$ of the Sturm-Liouville eigenvalue equation $$-\varphi''+q\varphi=\lambda y\varphi,$$ in the space $L^2(\mathbb R,ydx)$. Since the Weyl $m$-functions are known to contain all the information concerning the spectral properties of the above equation, the determination of the evolution of some spectral characteristics will allow to determine solutions of the hierarchies of evolution equations. 
\end{abstract}

\maketitle

\section{Introduction}  
The Korteweg-de Vries (K-dV) equation  $$q_t=\dfrac{3}{2}qq_x-\dfrac{q_{xxx}}{4},$$ as well as the Camassa-Holm equation 
$$4u_t-u_{xxt}=12uu_x-2u_xu_{xx}-uu_{xxx}+4\omega_0 u_x$$ are just two examples of evolution equations which can be studied by using instruments arising from the spectral theory of  Sturm-Liouville operators. In particular, solutions of the above mentioned equations can be determined when the initial data are related to the so-called \it potentials \rm of a Sturm-Liouville operator, whose spectrum has some prescribed properties. Well known examples in the case of the K-dV equation are given by the decaying potentials \cite{GGKM,Le1,SW}, the reflectionless potentials \cite{Cr,GWZ,JZ3,JZ4,Lu} and their generalizations \cite{Lu,Ko,Ma1}, the algebro-geometric potentials \cite{DMN,Mi,Z2} and their limits \cite{GN, Le,Lu,Mar}, the Sato-Segal-Wilson potentials \cite{SW}, and, of course, other types of potentials of the Schr\"odinger operator $$\mathcal S:=-D^2+q,$$ acting in $L^2(\mathbb R)$. 

More recently, a considerable effort has been made in order to extend the above-mentioned types of potentials to the more general case of the Sturm-Liouville operator $$\mathcal L=\dfrac{1}{y}(-D^2+q),$$ acting on $L^2(\mathbb R,ydx)$, with the purpose of both extending the known theory to a more general situation, and to determine solutions of evolution equations which generalize both the K-dV and the Camassa-Holm equation. In particular, when $q\equiv 1$ in $\mathcal L$,  it is possible, in some cases, to determine solutions of the Camassa-Holm equation by using the spectral theory of the operator, in a manner analogous to how the K-dV equation is treated via the Schrödinger operator
This has been done in the case of algebro-geometric potentials \cite{AF,CH,GH,Z1}, and, partially, in the case of decaying potentials \cite{Con1,Con2,Con3,Z3} (see also \cite{Bon,Liv} for other generalizations).

In recent years, a general hierarchy of evolution equations, referred to as the \bf Sturm-Liouville hierarchy\rm, has been introduced which encompasses both the K-dV and Camassa–Holm hierarchies, and introduces interesting generalizations. Solutions to this hierarchy have been developed in relation to algebro-geometric potentials and certain of their compact-uniform limits associated with the operator $\mathcal L$ \cite{JZ4bis,JZ5,JZ6,JZ7}. Furthermore, solutions have also been obtained in the framework of scattering theory \cite{Z3}.

This paper has two main goals. The first is to provide  a definition of the Sturm-Liouville hierarchy of evolution equations by using the Weyl $m$-functions of the underlying Sturm-Liouville operator $\mathcal L$. Since the Weyl $m$-functions encode essentially all the spectral information of the associated operator,  studying their evolution under the flows generated by the hierarchy yields profound insights into the \it  isospectral nature \rm of the dynamics. Consequently, important information about the corresponding solutions can be extracted from the behavior of the Weyl $m$-functions.

The second goal—closely related to the first—is to extend the classes of potentials that generate solutions to the equations in the Sturm–Liouville hierarchy, using asymptotic expansions. In particular, we will identify admissible expansions of the Weyl $m$-functions, which will in turn yield new families of potentials. These potentials can be used as initial data for solving the hierarchy of evolution equations, thereby generating new isospectral flows.

This paper is the first in a series of three, in which we aim to conduct as complete a study as possible of both the admissible potentials and the corresponding solutions to the hierarchy of evolution equations. As mentioned, the present work focuses on the formulation of the hierarchies via the Weyl $m$-functions, and on the identification of the classes of potentials that lead to well-defined solutions.

The paper is organized as follows: in Section 2 we give all the necessary background material concerning the spectral problems we will study. In Section 3 we will define and discuss some properties of the Sturm-Liouville hierarchy of evolution equations by using the \it zero-curvature \rm method. In Section 4,  we will define the equations we are willing to solve, we provide some cases and examples and we obtain some preliminary results, which are important both from the point of view of the hierarchy, and of the associated spectral problems. Section 5 contains the discussion of the motion of the Weyl $m$-functions under the flow induced by the hierarchy of evolution equations, allowing to relate the definition given by means to the zero-curvature method to that obtained by considering  the time evolution of the Weyl $m$-functions. In Section 6, we will determine the hierarchy using a certain expansion of the Weyl $m$-functions, thus producing corresponding families of potentials of a Sturm-liouville operator which will give rise to solutions of the hierarchy. Finally, in Section 7, we will give concrete examples of potentials which satisfy the expansion of the Weyl $m$-functions.

\section{Preliminaries}

In this section we review some general properties concerning the Sturm-Liouville operator on the line. The basic material we will state here can be essentially found in \cite{FJZ,JZ1}. Let us start by introducing the spectral problem we will deal with. Let $$\mathcal E_2=\{a=(q,y):\mathbb R\rightarrow\mathbb R^2\;|\; a \;\mbox{is uniformly continuous and bounded and}$$
$$\; 0<\delta<\inf_{x\in\mathbb R}y(x)\},$$ equipped with the standard topology of uniform convergence on compact sets. For $a\in\mathcal E_2$, let us define the operator $\mathcal L_a:\mathcal D\rightarrow L^2(\mathbb R,ydx)$ as $$\mathcal L_a(\varphi)=\dfrac{1}{y}\left[-\varphi''+q\varphi\right].$$  The domain $\mathcal D$ is given by $$\mathcal D=\{\varphi:\mathbb R\rightarrow\mathbb R\;|\;\varphi'\;\mbox{exists and is absolutely continuous,}\;\varphi''\in  L^2(\mathbb R,ydx)\}.$$ In the standard terminology, $q$ is a \it potential\rm, while $y$ is a \it density. \rm It is well-known that $\mathcal L_a$ admits a unique self-adjoint extension to all $L^2(\mathbb R,ydx)$. We will continue to denote such extension by $\mathcal L_a$. The most known examples of operators of this kind are the Schr\"odinger operator $$\mathcal S:=-D^2+q,\;\;\;a=(q,1)$$ and the acoustic operator $$\mathcal S_1:=\dfrac{1}{y}[-D^2+1],\;\;\;a=(1,y),$$ where $D$ denotes the differentiation operator. Let us denote by $\Sigma_a$ the spectrum of the operator $\mathcal L_a$. Then $\Sigma_a\subset \mathbb R$ is bounded below, unbounded above and invariant under translation of $a\in\mathcal E_2$. The set $R_a:=\mathbb R\setminus\Sigma_a$ is at most a countable union of disjoint (possibly unbounded) open intervals. 

For general elements $a\in\mathcal E_2$ it can be very hard to obtain more information concerning the structure of the set $\Sigma_a$, and in fact one often has to restrict to proper subsets of $\mathcal E_2$, for instance the \it periodic \rm or \it almost periodic \rm elements. The operator $\mathcal S$ has been extensively studied since many years, mainly due to the motivations explained in the introduction, while  $\mathcal S_1$ has been focused in the last three decades, the essential motivation being the relation with the Camassa-Holm equation, as we noted before. However, since many interesting phenomena can occur when one considers a general $a\in\mathcal E_2$, it is absolutely worthwhile to apply for the analysis of the general $\mathcal L_a$ (or even a more general case, as in \cite{FJZ,JZ1,JZ2}. Here, however, we will deal only with a sub-case, to avoid notational and computational difficulties). 

Now, we will introduce a Bebutov-type construction which has already revealed to be very useful when studying the spectral properties of $\mathcal L_a$ (see \cite{FJZ,JZ1,JZ2}). 
 Let $\tau_x(a)(\cdot)=a(x+\cdot)$ be the translation flow on $\mathcal E_2$, and choose an element $a_0\in \mathcal E_2$.  Consider the set $\mathcal A=cls\;\;Hull(a_0)=cls\;\{\tau_x(a_0)\;|\;x\in\mathbb R\}.$ Then $\mathcal A$ is a compact, translation invariant subset of $\mathcal E_2$. For $a\in\mathcal E_2$, we have the eigenvalue equation $$E_a(\varphi,\lambda):=-D^2\varphi+q\varphi=\lambda y\varphi,$$ which can be expressed in matrix form as $$\begin{pmatrix}\varphi\\ \varphi'\end{pmatrix}'=\begin{pmatrix}0 & 1 \\q-\lambda y&0\end{pmatrix}\begin{pmatrix}\varphi\\ \varphi'\end{pmatrix}.$$ In the remaining of this section, we will abuse notation and write $$a(x,\lambda)=\begin{pmatrix}0 & 1 \\q(x)-\lambda y(x)&0\end{pmatrix}.$$ 

Now, define a function $A:\mathcal E_2\rightarrow \mathbb M(2,\mathbb R)$ by $A(a)= a(0,\lambda)$. Applying the translation flow we have $a(x,\lambda)=A(\tau_x(a)).$ If we fix an element $a_0\in\mathcal E_2$ and construct the set $\mathcal A=cls\;Hull(a_0)$
as above, it makes sense to speak of the family of differential systems \begin{equation}\label{fam} \begin{pmatrix}\varphi\\ \varphi'\end{pmatrix}'=A(\tau_x(a))\begin{pmatrix}\varphi\\ \varphi'\end{pmatrix},\;\;\;a\in\mathcal A.\end{equation} Let $\Phi_a(x)$ be the fundamental matrix solution of \eqref{fam}, so that $\Phi_a(0)=Id$ for every $a\in\mathcal A$. We now introduce the fundamental concept of exponential dichotomy.

\begin{defin} The family \eqref{fam} is said to have an \underline{exponential dichotomy} over $\mathcal A$ if  there are
positive constants $\eta,\delta$ together with a continuous,
projection-valued function $P:\mathcal A\rightarrow\mathbb M_2$ (thus
$P(a)^2=P(a)$ for all $a\in\mathcal A$) such that the
following estimates hold:
\begin{itemize}
\item[(i)] $|\Phi_{a}(t)P(a)\Phi_{a}(s)^{-1}|\leq\eta
e^{-\delta(t-s)},\s t\geq s,$
\item[(ii)] $|\Phi_{a}(t)(I-P(a))\Phi_{a}(s)^{-1}|\leq\eta
e^{\delta(t-s)},\s t\leq s.$
\end{itemize}\finedim
\end{defin}

Perhaps the most important connection between the exponential dichotomy property and the spectrum of $\mathcal L_a$ is the following Theorem (see \cite{JM}, also \cite{FJZ}):
\begin{thm}Let $\mathcal A$ be defined as above and let us consider the family \eqref{fam}. Let $a\in \mathcal A$ have dense orbit. Then the spectrum $\Sigma_a$ of the operator $\mathcal L_a$ equals the set $$\Sigma_{ed}:=\{\lambda\in\mathbb C\;|\;\mbox{the family \eqref{fam} does \underline{not} admit an exponential dichotomy over} \mathcal A\}.$$
\end{thm}

Another consequence of the exponential dichotomy is that it permits to define the Weyl $m$-functions in a very convenient way. Let us now fix the boundary condition $\varphi(0)=0$. Then  define the self-adjoint operators $\mathcal L^\pm_a:L^2(\mathbb R^\pm,ydx)\rightarrow L^2(\mathbb R^\pm,ydx)$ on the half-lines $\mathbb R^\pm$ respectively.  For $\Im\lambda\neq 0$, the Weyl $m$-functions $m_\pm(a,\lambda)$ are usually defined as those complex numbers such that $$m_\pm(a,\lambda)=\dfrac{D\varphi_a^\pm(0,\lambda)}{\varphi_a^\pm(0,\lambda)},$$ where $\varphi_a^\pm(x,\lambda)$ are the unique (up to a constant multiple) solutions of the equation $E_a(\varphi,\lambda)$ which lie in $L^2(\mathbb R^\pm,ydx)$ respectively. Now, if $\Im\lambda\neq 0$ and $a\in\mathcal A$, then both $\Ker P(a)$ and $\IM P(a)$ are complex lines in $\mathbb C^2$. It can be proved that these lines can be parametrized as follows $$\IM P(a)=Span\begin{pmatrix}1 \\ m_+(a,\lambda)\end{pmatrix}\s \Ker P(a)=Span\begin{pmatrix} 1\\ m_-(a,\lambda)\end{pmatrix}.$$ We adopt the convention that $m_\pm(a,\lambda)=\infty$ if and only if $\IM P(a)=\begin{pmatrix}0\\1\end{pmatrix}$ and $\Ker P(a)=\begin{pmatrix}0\\1\end{pmatrix}$ respectively.
It is well-known that this can occur only when $\Im\lambda=0$, and in fact when $\varphi_a^\pm(x,\lambda)$ is an eigenfunction of $\mathcal L_a^\pm$ and hence $\lambda$ is an eigenvalue of $\mathcal L_a^\pm$ respectively. 
The Weyl $m$-functions admit nontangential limits $$m_\pm(a,\eta):=\lim_{\ep\rightarrow \infty}m_\pm(a,\eta+i\ep)$$ for a.a. $\eta\in\mathbb R$.

 If we use the translation flow, then we can define functions $$m_\pm:\mathbb R\times(\mathbb C\setminus\mathbb R)\rightarrow\mathbb C:(x,\lambda)\mapsto m_\pm(\tau_x(a),\lambda).$$ We will abuse notation and write $m_\pm(x,\lambda)$ instead of $m_\pm(\tau_x(a),\lambda)$ when there is no confusion.
It is easy to prove that, by their very definition, the Weyl $m$-functions  of $m_\pm(x,\lambda)$ satisfy the Riccati equation \begin{equation}\label{ricc2} m'+m^2=q-\lambda y.\end{equation}

 So, if $\Im\lambda\neq 0$, $m_\pm(x,\lambda)$ span $\IM P(\tau_x(a))$ and $\Ker P(\tau_x(a))$ respectively. If, for $\Im\lambda\neq 0$, we define $$\varphi_\pm(x,\lambda)=\exp\left(\int_0^x m_\pm(s,\lambda)ds\right),$$ then $\varphi_\pm(0,\lambda)=1$, $\varphi'_\pm(0,\lambda)=m_\pm(a,\lambda)$ and $\varphi_\pm$ satisfy the eigenvalue equation $E_a(\varphi,\lambda).$ By definition, $\varphi_\pm(\cdot,\lambda)\in L^2(\mathbb R^\pm,ydx)$ respectively. 

\bigskip

Now, we examine in more detail the behavior of the Weyl $m$-functions in proximity of the real line. We content ourselves to state the main results and give a picture of the situation; we  will indeed use the properties of $m_\pm(x,\lambda)$ to give an alternative interpretation/definition of the hierarchies of evolution equations we will study, which can be fruitful in future developments.

The observations above, together with other stuff (see, i.e., \cite{FJZ,JZ1}) allow us to state the following 
\begin{prop} The Weyl $m$-function $m_+(x,\lambda)$ (resp. $m_-(x,\lambda)$) extends holomorphically through every open interval $J\subset \mathbb R\setminus \Sigma_{\tau_x(a)}^+$ (resp. $J\subset \mathbb R\setminus \Sigma_{\tau_x(a)}^-$), where $\Sigma_{\tau_x(a)}^\pm$ denote the spectra of the operators $\mathcal L^\pm_{\tau_x(a)}$ respectively. In particular, $m_+(x,\lambda)$ (resp. $m_-(x,\lambda)$) has a pole  at every isolated eigenvalue of the operator $\mathcal L_{\tau_x(a)}^+$ (resp. $\mathcal L_{\tau_x(a)}^-$). \end{prop} Notice that, although the spectrum $\Sigma_a$ of $\mathcal L_a$ is invariant under the flow $\tau_x$, the spectra of $\mathcal L_{\tau_x(a)}^\pm$ are not in general, since the isolated eigenvalues may vary with respect to $x\in\mathbb R$.

Next, the diagonal zero-value of the Green's function $\mathcal G_a(\lambda)$ for the operator $\mathcal L_a$ can be defined as $$\mathcal G_a(\lambda):=\dfrac{y(0)}{m_-(a,\lambda)-m_+(a,\lambda)}\s(\Im\lambda\neq 0)$$ and corresponds to the value $\mathcal G_a(0,0,\lambda)$ of the classical Green's function for the operator $\mathcal L_a$. Actually $\mathcal G_a(\lambda)$ admits nontangential limits $$\mathcal G_a(\eta):=\lim_{\ep\rightarrow \infty}\mathcal G_a(\eta+i\ep)$$ for a.a. $\eta\in\mathbb R$. Acting by the translation, we have the function $$\mathcal G(x,\lambda):=\dfrac{y(x)}{m_-(x,\lambda)-m_+(x,\lambda)},$$ which represents the diagonal zero-value of the Green's function of the operator $\mathcal L_{\tau_x(a)}.$ Moreover, it is easy to show that $\mathcal G(x,\lambda)$ equals the diagonal Green's function $\mathcal G(x,x,\lambda)$ of the operator $\mathcal L_a$.
The behavior of the Green's function at real values determines particularly important subsets of $\mathcal E_2$ (see \cite{FJZ,JZ1}). 

\bigskip

Here and throughout all the paper, we will consider elements $a\in\mathcal E_2$ which satisfy the following basic assumption
\begin{hy1}\label{H1}

The absolutely continuous spectrum $\Sigma^{(ac)}_a\subset \Sigma_a$ of $\mathcal L_a$ contains a half-line $[\lambda_0,\infty)\subset\mathbb R$, for some $\lambda_0\in\mathbb R$.\end{hy1}


\section{The Zero-curvature formalism}

In this section we discuss the zero-curvature formulation for the hierarchies of evolution equations we will study.  As before, let $$A=\begin{pmatrix}0&1\\ \\q-\lambda y&0\end{pmatrix}$$ be the Sturm-Liouville matrix associated to the equation \begin{equation}\tag{$SL$}\label{31}\begin{pmatrix}\varphi\\ \\\varphi'\end{pmatrix}'=A\begin{pmatrix}\varphi\\ \\\varphi'\end{pmatrix},\end{equation} where $\varphi\in \mathcal D$, and the pair $(p,y)\in\mathcal E_2$.

We now introduce a time-dependence of the system, so that the variable $t$ is viewed as a parameter. Hence, from now on, $A$ depends on $x,t$ and $\lambda$, i.e., $$A(x,t,\lambda)=\begin{pmatrix}0&1\\ \\q(x,t)-\lambda y(x,t)&0\end{pmatrix}.$$
For the moment, we continue to assume that for every $t\in\mathbb R$, the pairs $(q(t,\cdot),y(t,\cdot))\in\mathcal E_2$.
Clearly, the equation \eqref{31} depends on the parameter $t$, so the solutions $\varphi$ do as well, obtaining a family \begin{equation}\tag{$SL_t$}\label{31t}\begin{pmatrix}\varphi(x,t,\lambda)\\ \\\varphi'(x,t,\lambda)\end{pmatrix}'=A(x,t,\lambda)\begin{pmatrix}\varphi(x,t,\lambda)\\ \\\varphi'(x,t,\lambda)\end{pmatrix}.\end{equation}
It is obvious that \eqref{31t} is equivalent to the family of eigenvalue equations \begin{equation}\label{E}-\varphi''(x,t)+q(x,t)\varphi=\lambda y(x,t)\varphi(x,t)\end{equation} and to the family of  Sturm-Liouville operators \begin{equation}\label{sl} \dfrac{1}{y(\cdot,t)}(-D^2+q(\cdot,t))\end{equation} in $L^2(\mathbb R,y(\cdot,t)dx)$.

Now, let us view $A$ as acting of the phase space $(\varphi,\varphi')$. So, if $\Phi$ denotes the column $\Phi=\begin{pmatrix}\varphi\\ \varphi'\end{pmatrix}$, then $A(\Phi)=A\Phi.$ Abusing slightly the notation, we denote by $D$ is the differentiation operator also in the phase space, i.e., $$D=\begin{pmatrix}\dfrac{\partial}{\partial x}&0\\ \\ 0&\dfrac{\partial}{\partial x}\end{pmatrix}$$ (the context will make clear the meaning of $D$). Let us consider, in the phase space, the operator $$L=D-A.$$ Note that, if $\Phi$ is a solution of the equation \eqref{31}, then $L(\Phi)=L\Phi=0.$ 

From now on, we will use the notation $u_t$, $u_x$, $u_{tt}$, $u_{xx}$ etc., to denote partial derivatives, together with the easier notation $u'$ and $\dot{u}$ for the standard derivatives with respect to $x$ and $t$ respectively. The usage will be determined by the context.

Define the operator  $$A_t=\begin{pmatrix}0&0\\ \\q_t-\lambda y_t&0\end{pmatrix},$$ and assume that there exists a matrix $B=B(x,t,\lambda)$ such that   the so-called \bf zero-curvature relation \rm holds: \begin{equation}\label{zc} A_t-B_x+[A,B]=0,\end{equation} where $[A,B]=AB-BA$ is the commutator of $A$ and $B$. This relation has several interesting consequences. First of all, from 
$$[L,B]=(D-A)B-B(D-A)=DB-AB-BD-BA=B_x-[A,B],$$ it follows that \begin{equation}\label{re1}A_t=[L,B].\end{equation}

Another fundamental consequence of \eqref{zc} is that the spectrum of $\mathcal L_a$ does not vary with respect to $t\in\mathbb R$, and in fact we have the following 

\begin{thm}\label{inv} Let $A(x,t,\lambda)$ as  above. Assume that there exists a matrix $B(x,t,\lambda)$ such that the zero-curvature relation \eqref{zc} holds.  Then the eigenvalues of the equation \eqref{E} {\bf do not} vary with $t$, i.e.,  $\lambda_\ast$ is an eigenvalue for the equation \eqref{E} for $t=0$ if and only if $\lambda_\ast$ is an eigenvalue of \eqref{E} for every $t\in\mathbb R$. Moreover, if $\psi$ is the corresponding eigenfunction then ,  $$\psi_t=T\psi+U\psi_x+c(t,\lambda)\psi,$$ where $c(t,\lambda)$ arises as a constant of integration. The above relation in turn implies that, if $\Psi=\begin{pmatrix}\psi\\\psi'\end{pmatrix}$, then   \begin{equation}\label{re10}\Psi_t=B\Psi+c(t,\lambda)\Psi, \end{equation}at every eigenvalue.\end{thm} 

\begin{proof}

Let us assume that $\psi(x,t)$ is an eigenfunction of \eqref{E}. Then, a priori, as $t$ varies, also $\lambda$ can change, i.e., $\lambda=\lambda(t)$. If $\Psi=\begin{pmatrix}\psi \\ \psi'\end{pmatrix}$, then $L\psi=0$. We have $$0=\dfrac{\partial}{\partial t}(L\Psi)=\dfrac{\partial L}{\partial t}\Psi+L\dfrac{\partial\Psi}{\partial t}=-\dfrac{\partial A}{\partial t}\Psi+L\dfrac{\partial\Psi}{\partial t}=-A_t\Psi-\mathcal Y\Psi+L\dfrac{\partial\Psi}{\partial t},$$ where $$\mathcal Y=\begin{pmatrix}0&0\\ \\\dfrac{d\lambda}{dt}y&0\end{pmatrix}.$$ It follows that, using \eqref{re1},
$$\mathcal Y\Psi=-[L,B]\Psi+L\dfrac{\partial\Psi}{\partial t}=-LB\Psi+L\dfrac{\partial\Psi}{\partial t},$$ hence \begin{equation}\label{re2}\mathcal Y\Psi=L\left(\dfrac{\partial\Psi}{\partial t}-B\Psi\right)\end{equation} Now, notice that, if $\Phi=\begin{pmatrix}\varphi \\ \varphi'\end{pmatrix}$, then $L$ maps $\Phi$ into  $$L\Phi=L\begin{pmatrix}\varphi\\ \varphi'\end{pmatrix}=\begin{pmatrix}0\\ \\ S(\varphi)\end{pmatrix},$$ where $$S(\varphi)=\varphi''-(q-\lambda y)\varphi.$$ Thus, if we write $$B=\begin{pmatrix}T&U\\ V&Z\end{pmatrix},$$ we can rewrite \eqref{re2} as follows \begin{equation}\label{re3}\dfrac{d\lambda}{dt}y\psi=S(\dot{\psi}-T\psi-U\psi').\end{equation} Now, the operator $S$ is self adjoint in $L^2(\mathbb R)$, and \begin{equation}\label{re5}\dfrac{d\lambda}{dt}\langle y\psi,\psi\rangle= \left\langle S(\dot{\psi}-T\psi-U\psi'),\psi\right\rangle=\left\langle \dot{\psi}-T\psi-U\psi',S(\psi)\right\rangle=0,\end{equation} hence $$\dfrac{d\lambda}{dt}||\sqrt y\cdot\psi||_2=0\Rightarrow \dfrac{d\lambda}{dt}=0.$$  Notice also that, from \eqref{re3}, since $\dfrac{d\lambda}{dt}=0$, we have that the function $\dot{\psi}-T\psi-U\psi'$ is an eigenfunction with the same $\lambda$ as eigenvalue, hence $$\dot{\psi}-T\psi-U\psi'=c(t,\lambda)\psi.$$ \end{proof}

\begin{rem} It is important to highlight that  the relations \eqref{re3} and \eqref{zc} remain valid if we add to $B$ the identity operator times a function $c(t,\lambda)$. \end{rem}

Now, we try to determine the structure of the matrix  $B(x,t,\lambda)$.
Write, as above, $$B=\begin{pmatrix}T&U\\ \\V&Z\end{pmatrix}$$ and use \eqref{zc} and the above observations to show that \begin{equation}\label{sis}\begin{cases}T_x=-Z_x
\\ U_x=-2Z=2T\\ -Z_x+(q-\lambda y)U-V=0\\q_t-\lambda y_t-V_x-U_x(q-\lambda y)=0\end{cases}\end{equation}

These relations imply that $$B=\begin{pmatrix}-\dfrac{U_x}{2}+c(t,\lambda) & &U\\ \\-\dfrac{U_{xx}}{2}+(q-\lambda y)U& & \dfrac{U_x}{2}+c(t,\lambda)\end{pmatrix},$$ together with the \bf compatibility condition \rm \begin{equation}\label{cc}q_t-\lambda y_t+\dfrac{U_{xxx}}{2}-2U_x(q-\lambda y)-U(q-\lambda y)_x=0,\end{equation}
where $U=U(x,t,\lambda)$ is a function whose regularity will be declared in the following. 

Note that with this notation, if $\varphi$ is a solution of \eqref{E},  then \begin{equation}\label{comp1}\begin{cases}-\varphi_{xx}+q\varphi=\lambda y\varphi\\ \\ \varphi_t=-\dfrac{U_x}{2}\varphi+U\varphi_x+c(t,\lambda)\varphi.\end{cases}\end{equation}

The compatibility condition \eqref{cc} is of fundamental importance. It will determine the hierarchy of evolution equations we are going to study. Briefly speaking, the existence of a function $U(x,t,\lambda)$ which defines $B(x,t,\lambda)$ is strictly related to two evolution equations. Let us make some simple examples. 

\begin{ex}
\rm 

{(\bf 1)}\;\; If we assume that $U(x,t,\lambda)=u(t,x)$, i.e., $U(x,t,\lambda)$ is a function which does not depend on $\lambda$, then \eqref{cc} translates to $$q_t-\lambda y_t+\dfrac{u_{xxx}}{2}-2u_x(q-\lambda y)-u(q_x-\lambda y_x)=0,$$ and equating the powers of $\lambda$ we obtain $$\begin{cases}q_t+\dfrac{u_{xxx}}{2}-2u_xq-uq_x=0\\ \\-y_t-2u_xy+uy_x=0.\end{cases}$$ Now, if $u(x,t)$ is fixed, then we obtain two evolution equations, one for $q$ and another for $y$. Take, for instance, $u(x,t)=c$, $c\in\mathbb R$. Then we have $$\begin{cases}q_t-cq_x=0\\ \\y_t+cy_x=0\end{cases},$$ which clearly gives, if $q(x,0)=q_0(x)$ and $y(x,0)=y_0(x)$, $$\begin{cases}q(x,t)=q_0(x+ct)\\ y(t,x)=y_0(x-ct)\end{cases}.$$ Note that, in this case $$B=\begin{pmatrix}0&c\\ \\c(q-\lambda y)&0\end{pmatrix},$$ which express the fact that the solutions $\Phi(x,t)$ of \eqref{31t} are given by $\Phi_0(x+ct)$ where $\Phi_0(x)$ is a solution of  \eqref{31}. 

{(\bf 2)}\;\;Another interesting case  can be obtained when we \it fix \rm the evolution of one between $q$ and $y$. In this case, the corresponding system has as unknowns both $u$ and the other between $q$ and $y$. Let us again consider the previous case. If we fix a function $y(x,t)>\delta>0$, then we can retrieve $u$ by solving a first order ordinary equation, namely $$u_x=\dfrac{y_x}{2y}u-\dfrac{y_t}{2y}.$$ Then the function $q$ must solve an evolution equation, which is the compatibility condition for the existence of the matrix $B$, hence for the validity of the zero-curvature relation. In our example, we easily obtain $$u(x,t)=\sqrt{y(x,t)}\left(c(t)-\int_0^x\dfrac{y_t(s,t)}{2y^{3/2}(s,t)}ds\right),$$ and putting this quantity in the equation for $q$ we have the desired evolution equation. As an example, assume that $y(x,t)=y(x)$, then $$u(x,t)=c(t)\sqrt{y(x)},$$
  and \begin{equation}\label{ex}q_t=c(t)\left(2(\sqrt{y})_xq+\sqrt yq_x-\dfrac{1}{2}(\sqrt y)_{xxx}\right).\end{equation} As we will see in the following, the equation \eqref{ex} admits a special stationary solution, namely if we set $z=\dfrac{1}{y}$, then the function of the only $x$-variable $$\tilde z(x)=-\left(\dfrac{z'}{4z}\right)'+\left(\dfrac{z'}{4z}\right)^2=-\dfrac{z''}{4z}+\dfrac{5}{16}\left(\dfrac{z'}{z}\right)^2$$ solves the equation \eqref{ex}, since $\tilde z_t=0$. As another example, let us fix $y(x,t)=1$. Then $u(x,t)=u_0(t)$ and the equation for $q$ reads as $$q_t-u_0(t)q_x=0,$$ which can be easily solved to give $$q(x,t)=q(x+U_0(t),0),$$ where $$U_0(t)=\int_0^t u_0(s)ds.$$ \finedim
  \end{ex}

\bigskip

 It is quite clear that, a priori, $U(x,t,\lambda)$ can be any sufficiently regular function. But then it would be unclear how to recover the compatibility condition in a form which does not depend on the parameter $\lambda$. Before explaining the nature of the function $U(x,t,\lambda)$ we will consider, it is convenient to observe that we can replace the zero-curvature relation \eqref{zc} with \begin{equation}\label{zc1}\eta(\lambda)A_t-B_x+[A,B]=0,\end{equation} where $\eta(\lambda)$ is any continuous function of the variable $\lambda\in\mathbb C$. In this case, the compatibility conditions \eqref{cc} and  \eqref{comp1} translate to
 
\begin{equation}\label{cc1}\eta(\lambda)\left(q_t-\lambda y_t\right)+\dfrac{U_{xxx}}{2}-2U_x(q-\lambda y)-U(q-\lambda y)_x=0,\end{equation}
 and 
  \begin{equation}\label{comp2}\begin{cases}-\varphi_{xx}+q\varphi=\lambda y\varphi\\ \\ \eta(\lambda)\varphi_t=-\dfrac{U_x}{2}\varphi+U\varphi_x+c(t,\lambda)\varphi.\end{cases}\end{equation}
 respectively. For clarity, let us summarize these facts into the following 
 
 \begin{prop}\label{P1} The zero-curvature relation \eqref{zc1} $$\eta(\lambda)A_t-B_x+[A,B]=0$$ is valid if and only if there exists a function $U(x,t,\lambda)$ and a function $\eta(\lambda)$ such that $$B(x,t,\lambda)=\begin{pmatrix}-\dfrac{U_x}{2}+c(t,\lambda) & &U\\ \\-\dfrac{U_{xx}}{2}+(q-\lambda y)U& & \dfrac{U_x}{2}+c(t,\lambda)\end{pmatrix}$$ and the compatibility condition \eqref{cc1} $$\eta(\lambda)\left(q_t-\lambda y_t\right)+\dfrac{U_{xxx}}{2}-2U_x(q-\lambda y)-U(q-\lambda y)_x=0$$ hold simultaneously. This implies that for every solution $\varphi$ of the equation \eqref{E}
 $$-\varphi_{xx}+q\varphi=\lambda y\varphi$$  we have $$\eta(\lambda)\varphi_t=-\dfrac{U_x}{2}\varphi+U\varphi_x+c(t,\lambda)\varphi,$$ where $c(t,\lambda)$ is as in Theorem \ref{inv}. The eigenvalues of the spectral problem $$-\varphi_{xx}+q\varphi=\lambda y\varphi$$ on $L^2(\mathbb R,ydx)$ do not depend on $t\in\mathbb R$.\end{prop} 
 
 \section{The Sturm-Liouville hierarchy of evolution equations}

It is the moment to declare how to take $\eta(\lambda)$ and $U(x,t,\lambda)$ in order to have a well-determined compatibility relation. We take $U(x,t,\lambda)$ to be a polynomial of degree $r$ in the variable $\lambda\in \mathbb C$, whose coefficients depend on $x$ and $t$: $$U(x,t,\lambda)=\sum_{j=0}^r u_j(x,t)\lambda^j.$$ The coefficients $u_j(x,t)$ ($j=0,\dots,r$) are continuous functions in the variable $t$ for every fixed $x\in\mathbb R$ and of class $C^3$ in the variable $x$ for every fixed $t\in\mathbb R$. Moreover $\eta(\lambda)=\lambda^k$, where $0\leq k\leq r.$ It is convenient to write down these assumptions as follows:
 \begin{hy}\label{H} For a given system  \eqref{comp2},  \begin{equation*}\begin{cases}-\varphi_{xx}+q\varphi=\lambda y\varphi\\ \\ \eta(\lambda)\varphi_t=-\dfrac{U_x}{2}\varphi+U\varphi_x+c(t,\lambda)\varphi,\end{cases}\end{equation*} we assume that 
 
 \begin{enumerate}\item the pair $(q(x,t),y(x,t))\in\mathcal E_2$ for every $t\in\mathbb R$;
 \item $q,q_t,y,y_t\in C^{2r+1}(\mathbb R)\times C(\mathbb R)$;  \item the function $U(x,t,\lambda)$ is a polynomial $$U(x,t,\lambda)=\sum_{j=0}^r u_j(x,t)\lambda^j,$$ where the coefficients $u_j(x,t)$ ($j=0,\dots,r$) are continuous functions in the variable $t$ for every fixed $x\in\mathbb R$ and of class $C^{2r+1}$ in the variable $x$ for every fixed $t\in\mathbb R$; \item $\eta(\lambda)=\lambda^k$, where $0\leq k\leq r.$\end{enumerate}\end{hy}
 Note that Hypotheses \ref{H} can be clearly weakened in many ways. One is to consider $U$ to be an entire function (see \cite{JZ7,Z2}), another is to take $\eta(\lambda)$ to be any polynomial in $\lambda$ with constant coefficients (see \cite{JZ5,JZ6,Z1}).
 
 Let us make some concrete examples. 
\begin{ex}\rm
 \phantom{a}
 \begin{itemize} \item Take $y\equiv 1$,  $r=1$ and $k=0$. We have $U=u_1\lambda+u_0$ and \eqref{cc1} reads $$q_t=2(u_{1x}\lambda+u_{0,x})(q-\lambda)+(u_1\lambda+u_0)q_x-\dfrac{u_{1,xxx}\lambda+u_{0,xxx}}{2}.$$
 Ordering with respect to $\lambda$ we obtain $$\begin{cases}-2u_{1,x}=0\\ \\2u_{1x}q-2u_{0,x}+u_1q_x-\dfrac{u_{1,xxx}}{2}=0\\ \\q_t=2u_{0,x}q+u_0q_x-\dfrac{u_{0,xxx}}{2}.\end{cases}$$
 We can take $u_1=1$, so that $q_x=2u_{0,x}$. Hence if we take $q=2u_0$ we have that the above system is solvable if $q$ is the solution of the evolution  equation   $$q_t=\dfrac{3}{2}qq_x-\dfrac{q_{xxx}}{4},$$ which is the well-known Korteweg-de Vries (K-dV) equation. So, if $q$ satisfies the K-dV equation, then the zero-curvature relation is valid, and the eigenvalues of the eigenvalue equation \eqref{E} on $L^2(\mathbb R)$, which in this case can be written as $$-\varphi_{xx}+q\varphi=\lambda\varphi$$ (i.e., the standard Schr\acc odinger equation),
 do not vary with $t\in\mathbb R$.
 \item Let $q\equiv 1$, $k=r=1.$ Then we have $$\begin{cases}2u_{0,x}-\dfrac{u_{0,xxx}}{2}=0\\ \\2u_{1,x}-2u_{0,x}y-u_0y_x-\dfrac{u_{1,xxx}}{2}=0\\ \\y_t=2u_{1,x}y+u_1y_x.\end{cases}$$ So we can take $u_0=1$ and the system becomes $$\begin{cases}y_x=2u_{1x}-\dfrac{u_{1,xxx}}{2}\Rightarrow y(x,t)=2u_{1}(x,t)-\dfrac{u_{1,xx}}{2}+\omega(t)\\ \\ y_t=2u_{1x}y+u_1y_x.\end{cases}$$ Writing $u$ for $u_1$ and taking $\omega(t)$ as a constant $\omega(t)=\omega_0$, we obtain the equation $$2u_t-\dfrac{u_{xxt}}{2}=2u_x\left(2u-\dfrac{u_{xx}}{2}\right)+2\omega_0 u_x+u\left(2u_x-\dfrac{u_{xxx}}{2}\right), $$ or \begin{equation}\label{CH} 4u_t-u_{xxt}=12uu_x-2u_xu_{xx}-uu_{xxx}+4\omega_0 u_x,\end{equation} which is an equivalent form of  the Camassa-Holm equation \cite{Con1,Con2,Con3} (the original equation in the literature is obtained via a change of variables). So, if $u$ is a solution of the equation \eqref{CH}, and if $y$ is given from $u$ by $$y=2u-\dfrac{u_{xx}}{2}+\omega_0,$$ then the zero-curvature relation $\lambda A_t-B_x+[A,B]=0$ holds, and the eigenvalues of the equation $$-\varphi_{xx}+\varphi=\lambda y\varphi$$  on $L^2(\mathbb R,ydx)$ do not vary with $t\in\mathbb R$.  Note that we can take $\omega(t)$ to be not a constant, and obtain the same result by redefining $y(x,t)$ to be $y(x,t)-\omega(t)+\omega_0$. 
 
 \item Now take $y\equiv 1$, $r=2$ and $k=0$. Then we proceed as above and obtain the equation $$q_t=-\dfrac{1}{2}\left[\dfrac{-q_{xx}}{8}+\dfrac{3}{8}q^2\right]_{xxx}+q\left[-\dfrac{q_{xxx}}{4}+\dfrac{3}{2}qq_x\right]+q_x\left[-\dfrac{q_{xx}}{8}+\dfrac{3}{8}q^2\right],$$ which is the second order K-dV equation (see \cite{Le1,SW,Z2}).
 
 \item Take $q\equiv 1$, $r=2$ and $k=1$. Then we have, by choosing appropriate constants where possible, $$\begin{cases}u_0=1\\ \\ y=2u_1-\dfrac{u_{1,xx}}{2}\\ \\u_2=\dfrac{1}{\sqrt y} \\ y_t=2u_{1,x}y+u_1y_x+F(y),\;\;F(y)=\left(\dfrac{1}{2\sqrt y}\right)_{xxx}-2\left(\dfrac{1}{\sqrt y}\right)_x\end{cases},$$ which is a modified version of the  Camassa-Holm equation.

 \end{itemize}\end{ex}

\bigskip
 We now make some observations concerning the effects of the choice we made for $U$ and $\eta$, and we illustrate the general method to retrieve the evolution equation which establishes the compatibility condition for the zero-curvature \eqref{zc1} to hold. 

If we look at the examples above, it is clear that the evolution equation providing the compatibility condition for the zero-curvature to hold arises as an equation in a system of $r+2$ equations. We will describe this in detail.

Let us fix $k\leq r$. Then $U(x,t,\lambda)$ is a polynomial $$U(x,t,\lambda)=\sum_{j=0}^ru_j(x,t)\lambda^j.$$
The relation \eqref{cc1} becomes $$\lambda^k\left(q_t-\lambda y_t\right)=2(q-\lambda y)\left(\sum_{j=0}^ru_{j,x}\lambda^j\right)+
(q_x-\lambda y_x)\left(\sum_{j=0}^ru_j\lambda^j\right)-\dfrac{1}{2}\left(\sum_{j=0}^r u_{j,xxx}\lambda^j\right),$$ and comparing the powers of $\lambda$ we obtain the system ($j=1,\dots,r$) \begin{equation}\label{sys} \begin{cases}2u_{r,x}y+u_ry_x=\delta_{(r+1,k+1)}\cdot y_t\\ \vdots \\ 2u_{j,x}q-2u_{j-1,x}y+u_jq_x-u_{j-1}y_x-\dfrac{u_{j,xxx}}{2}=\delta_{(j,k)}\cdot q_t-\delta_{(j,k+1)}\cdot y_t\\ \vdots \\ 2u_{0,x}q+u_0q_x- \dfrac{u_{0,xxx}}{2}=\delta_{(0,k)}\cdot q_t,
 \end{cases}\end{equation} where, for $0\leq i,j\leq r$,  $\;\;\delta_{(i,j)}=\begin{cases}0,&i\neq j\\ 1,&i=j.\end{cases}$ 
 
 Let us analyze in detail what \eqref{sys} means. First of all, it is made of $r+2$ equations involving $q,y$, the coefficients $u_0,\dots,u_r$, and their derivatives. The time derivatives $q_t$ and $y_t$ appear only in two equations, namely the $k$-th and the $(k+1)$-th starting from below. So if, for instance, $k=0$, then $q_t$ appears in the first and $y_t$ in the second line from below. We need $r+1$ equations to recover recursively the coefficients $u_1,\dots,u_r$, while it remains only one equation to determine the time evolution of $q$ and $y$. It follows that we have to \bf fix \rm the time evolution of one between $q$ and $y$, and then use the equation containing explicitly the time evolution of the other function as the \bf compatibility condition \rm for the system \eqref{sys} to have a solution. So, the compatibility condition \eqref{cc1} is equivalent to require that one between $q$ and $y$ solve an evolution equation, defined via the remaining equation in the system \eqref{sys}.

 In view of the   above, we now give the following
\begin{defin}Let $0\leq k\leq r$ be fixed. \begin{itemize}\item If we fix a function $y(x,t)$ then the $r$-th equation in the Sturm-Liouville $(k,y)-$Hierarchy of evolution equations (briefly $H^q_r(k,y)$) is the evolution equation \begin{equation}\label{qH}\tag{$H^q_r(k,y)$} q_t=2u_{k,x}q+u_kq_x-2u_{k-1,x}y-u_{k-1}y_x-\dfrac{u_{k,xxx}}{2}.\end{equation}
\item If we fix a function $q(x,t)$ then the $r$-th equation in the Sturm-Liouville $(k,q)-$Hierarchy of evolution equations (briefly $H^y_r(k,q)$) is the evolution equation \begin{equation}\label{yH}\tag{$H^y_r(k,q)$} y_t=-2u_{k+1x}q-u_{k+1}q_x+2u_{k,x}y+u_{k}y_x+\dfrac{u_{k+1,xxx}}{2},\end{equation} with the convention that $u_{r+1}=0,$ for every chosen $r$.

\end{itemize}
\end{defin}

The main purpose of the paper, as already anticipated, is that of expressing  \eqref{sys} and the above evolution equations by using the properties of the Weyl functions, starting from some types of pairs $(q,y)\in\mathcal E_2$ which satisfy an additional assumption (besides Hypothesis \ref{H1}). This will allow to have a deep insight about  the spectral information of the related ($t$-dependent) family of Sturm-Liouville operators, and hence about the \it flow \rm determined by the solutions of the evolution equations.

\begin{rem}If $k=0$ and we fix $y(x,t)=1,$  $H^q_r(0,1)$ is the standard K-dV hierarchy \cite{Le1,Z2}. If, instead, we take $k=r$ and $q(x,t)=1$ we obtain the Camassa-Holm hierarchy $H^y_r(r,1)$ \cite{CH,Con2,GH,JZ4bis,Z1}. For example, if $k=r=2$ we  can take $$\begin{cases} u_0=1\\ \\ y=2u_{1}-\dfrac{u_{1,xx}}{2} \\ \\ 2u_2-\dfrac{u_{2,xx}}{2}=3u_1^2+\dfrac{1}{4}(u_{1,x})^2+\dfrac{u_1u_{1,xx}}{2}\\ \\ y_t=2u_{2,x}y+u_2y_x.\end{cases}$$ The last equation in the system is the second-order Camassa-Holm equation (in our notation, it is $H^y_2(2,1)$) \cite{Z1}. Notice that it can be expressed as an evolution equation for the function $u_1$ since both  $y$ and $u_2$ can be interpreted as functions of $u_1$ only.\end{rem}

The general method to derive the equation of the hierarchy is as follows. Fix $k\leq r$ and one between $q$ and $y$.  Use all the equations which do not contain $q_t$ and $y_t$ to determine recursively the coefficients $u_j$ for $j=0,\dots,k-1,k+2,\dots,r$, then use the two remaining equations: one (the one containing the time derivative of the function we have fixed) is used to determine recursively the remaining coefficient, the other furnishes the evolution equation of the hierarchy we are looking for. 

Before going forward, let us focus on an example which will have significance also in the following. First we prove the following 
\begin{lemma} Let  $y(x,t)$ be a function of class $C^3(\mathbb R)$ for every fixed $t\in\mathbb R$.  Define \begin{equation}\label{tq}\tilde q:=-\left(\dfrac{y_x}{4y}\right)_x+\left(\dfrac{y_x}{4y}\right)^2.\end{equation} Then, if $$\tilde u_0=\dfrac{1}{\sqrt y},$$ we have \begin{equation}\label{sta1} 2\tilde u_{0,x}(\tilde q+\lambda_0 y)+\tilde u_0(\tilde q_x+\lambda_0 y_x)-\dfrac{\tilde u_{0,xxx}}{2}=0,\end{equation} for every $\lambda_0\in\mathbb R$\end{lemma}
\begin{proof} It is a straightforward computation, after noticing that $$\tilde q=-\dfrac{y_{xx}}{4y}+\dfrac{5}{16}\left(\dfrac{y_x}{y}\right)^2.$$\end{proof}
We can prove the following 

 \begin{prop}\label{P101} Let $y$ do not depend on $t$. Then the function $\tilde q+\lambda_0 y$, where $\tilde q$ is defined in  \eqref{tq},  is a stationary solution of {\bf every} equation $H_r^q(k,y)$.
 
  Vice-versa, if $q$ does not depend on $t$ and if there exists a solution $y(x)$ of the Riccati-type equation $$q=-\left(\dfrac{y_x}{4y}\right)_x+\left(\dfrac{y_x}{4y}\right)^2+\lambda_0 y$$ for some $\lambda_0\in\mathbb R$, then such a $y$ is a stationary solution of {\bf every} equation $H^y_r(k,q)$.\end{prop}
\begin{proof}
Let  the function $y$  vary only with respect to $x$, i.e., $y=y(x)$, and construct the hierarchy with $r=1$ and $k=0$. We have \begin{equation}\label{s1}\begin{cases}2u_{1,x}y+u_1y_x=0\\ \\ 2u_{1,x}q+u_1q_x-2u_{0,x}y-u_0y_x-\dfrac{u_{1,xxx}}{2}=0\\ \\q_t=2u_{0,x}q+u_0q_x-\dfrac{u_{0,xxx}}{2}.\end{cases}\end{equation}
The first equation easily gives $$u_1=\dfrac{1}{\sqrt y},$$ hence $ u_1=\tilde u_0$ (see the above Lemma). This simple fact has two main consequences. The first is that the system is satisfied with the choice of $q=\tilde q+\lambda_0 y$ $(\lambda_0\in\mathbb R)$ and $u_0=\tilde u_0$. Indeed, let $q=\tilde q+\lambda_0 y$; then $q$ does not depend on $t$ and it is easily seen that the  equations in the system are satisfied, by using the above Lemma and the fact that $2\tilde u_{0,x}y+\tilde u_0y_x=0.$
The same result can be obtained for every $r$ and $k\leq r$, by choosing $u_0=u_1=\dots=u_r=\tilde u_0$ and $q=\tilde q+\lambda_0 y$. \end{proof}

  Proposition \ref{P101} provides some important facts.  We list some of them.
\begin{rem} \rm \phantom{a}\begin{itemize}\item If we keep in mind that the functions $q$ and $y$ are the potential and the density respectively of a Sturm-Liouville operator, then Proposition \ref{P101} gives us a fundamental Sturm-Liouville  operator, namely \begin{equation}\label{tL}\tilde L:=\dfrac{1}{y}\left(-D^2+(\tilde q+\lambda_0 y)\right).\end{equation} The general solutions of the associated eigenvalue equation \begin{equation}\label{tL} \tilde E:=-\varphi_{xx}+\left(\tilde q+\lambda_0 y\right)\varphi=\xi y\varphi,\end{equation}
where we have denoted the spectral parameter with $\xi$, are $$\varphi_\pm(x)=y^{-1/4}(x)e^{\pm i\sqrt{\xi-\lambda_0}\mathcal I(x)},$$ where $$\mathcal I(x)=\int_0^x\sqrt{y(s)}ds,$$ and it is tacitly assumed that, if $\xi-\lambda_0<0$, then $\sqrt{\xi-\lambda_0}=i\sqrt{\lambda_0-\xi}.$ 
It follows easily that $\tilde L$ has only absolutely continuous spectrum $\tilde\Sigma_\xi=[\lambda_0,+\infty)$. Moreover, by writing $\tilde E$ as $$\varphi_{xx}+\tilde q\varphi=(\xi-\lambda_0)y\varphi,$$ and setting $\lambda=\xi-\lambda_0$, the operator \begin{equation}\label{tLL} \tilde L_0:=\dfrac{1}{y}\left(-D^2+\tilde q\right)\end{equation} has only absolutely continuous spectrum $\tilde\Sigma_\lambda=[0,+\infty)$. A detailed discussion of these facts will be part of one forthcoming paper \cite{Z4}.

\item A second consequence of  the observation above is an interesting extension of the well-known K-dV equation.  Namely, imagine that we fix a function $y$ which does not depend on $t$, and take into account the functions $\tilde q$ and $\tilde u_0$ as above. Assume further, for simplicity, that $$\lim_{|x|\rightarrow+\infty}y(x)=1.$$ Let us take $k=0$ and write down the first order hierarchy, i.e., for $r=1$. Now we allow $q$ to vary with respect to $t$ as well, and define $$Q(x,t)=q(x,t)-\tilde q(x).$$ The the system \eqref{s1} gives $$\begin{cases}u_1=\tilde u_0=\dfrac{1}{\sqrt y}\\ \\ 2u_{1,x}Q-2u_{0,x}y+u_1Q_x-u_0y_x=0\\ \\q_t=-\dfrac{u_{0,xxx}}{2}+2u_{0,x}q+u_0q_x.\end{cases}$$ A direct computation of $u_0$ from the second equation gives $$u_0=\dfrac{1}{\sqrt y}\left[c+\dfrac{Q}{y}\right]=\dfrac{1}{\sqrt y}\left[c+\dfrac{q-\tilde q}{y}\right].$$ Taking $c=0$ for simplicity and substituting  into the evolution equation for $Q$ we obtain $$q_t=\left(\dfrac{q-\tilde q}{2y\sqrt y}\right)_{xxx}+2q\left(\dfrac{q-\tilde q}{y\sqrt y}\right)_x+q_x\left(\dfrac{q-\tilde q}{y\sqrt y}\right)$$
This equation appears as the generalization of the K-dV equation when the bottom has the variable slope $\tilde q(x)$ (see, for instance, \cite{IS}). The nature of the function $Q$ is very important, but is out of the goal of the present paper. We address the reader to \cite{Z4,Z5} for a detailed discussion.
 \end{itemize}\end{rem}
 \section{The motion of the  Weyl functions} 
 In this section we will determine the time evolution of the Weyl functions, and we will draw some significant relations with the coefficients of the recursion for the hierarchy. We already defined the Weyl $m$-functions $m_\pm(x,\lambda)$ in Section 2. Once that a time evolution has been introduced by means of the zero-curvature relation \eqref{zc1}, we try to determine their evolution equations. First of all, now we have functions $m_\pm(x,t,\lambda)$. They satisfy the Riccati equation $$m_x+m^2=q-\lambda y$$ for every fixed $t\in\mathbb R$ and for every $\lambda$ which is not an eigenvalue of the associated Sturm-Liouville operator. We start with the following 
 
\begin{prop}\label{P2} Let the zero-curvature relation \eqref{zc1} hold. Then the Weyl $m$-functions $m_\pm(x,t,\lambda)$ are defined for every $\lambda\in \mathbb C$, except at the eigevalues of the equation \eqref{E}, and  satisfy the system \begin{equation}\label{mxt} \begin{cases}m_x+m^2=q-\lambda y\\ \\ \eta(\lambda)m_t=-\dfrac{U_{xx}}{2}+U_xm-Um^2+U(q-\lambda y).\end{cases}\end{equation}\end{prop}
\begin{proof}
 Suppose that the zero-curvature relation \eqref{zc1}  holds. Then Proposition \ref{P1} gives us the time evolution of every solution of the eigenvalue equation \eqref{E}, i.e., the second equation in \eqref{comp2}. Let $\lambda\in\mathbb C$ be not an eigenvalue. Then the functions $$\psi_\pm(x,t,\lambda)=\exp\left(\int_0^x m_\pm(s,t,\lambda)ds\right)$$ are solutions of the equation \eqref{E} such that $$\dfrac{\psi_{\pm,x}(x,t,\lambda)}{\psi_\pm(x,t,\lambda)}=m_\pm(x,t,\lambda).$$ This follows easily from the Riccati equation for $m_\pm$.
 
Since the solutions $\psi_\pm$ satisfy the second equation in \eqref{comp1}, i.e., 
 $$\eta(\lambda)\varphi_t=-\dfrac{U_x}{2}\varphi+U\varphi_x+c(t,\lambda)\varphi,$$ we have, omitting the subscript $\pm$ for convenience, $$\eta(\lambda)m_t=\eta(\lambda)\left(\dfrac{\psi_{xt}}{\psi}-\dfrac{\psi_x\psi_t}{\psi^2}\right).$$ Using the above relation we obtain \begin{equation}\label{mt}\eta(\lambda)m_t=-\dfrac{U_{xx}}{2}+U_xm-Um^2+U(q-\lambda y).\end{equation}
 
 The equation \eqref{mt} gives the time evolution of the Weyl $m$-functions when the zero-curvature \eqref{zc1} holds.\end{proof}

 Now, assume that $\lambda=0$ is not in the spectrum of \eqref{sl} for every $t\in\mathbb R$.  (if this is not the case, we can argue by a translation argument). Then there are well-defined function $\tilde m_{\pm}(x,t):=m_\pm(x,t,0).$ They satisfy the system \eqref{mxt} for $\lambda=0$, which reads as $$\begin{cases}\tilde m_x+\tilde m^2=q\\ \\ \eta(0)\tilde m_t=-\dfrac{u_{0,xx}}{2}+u_{0,x}\tilde m-u_0\tilde m^2 +u_0q.\end{cases}$$ It is fairly clear that both $\tilde m_\pm$ do not depend on $y$.
 
 We now distinguish two cases: $(i)\;\; \eta(0)\neq 0$ and $(ii)\,\;\eta(0)=0.$ In the case $(i)$, then we can argue that $\tilde m_\pm$ satisfy \begin{equation}\label{mi} \eta(0)\tilde m_t=-\dfrac{u_{0,xx}}{2}+(u_0\tilde m)_x, \end{equation} while in the second case we have \begin{equation}\label{mii} 0=-\dfrac{u_{0,xx}}{2}+(u_0\tilde m)_x.\end{equation}
 Both the above equations express $u_0$ as functions of  $\tilde m_\pm$ and vice-versa.  In the case $(i)$ we have to solve a  PDE, while in $(ii)$ we have a linear second order ODE. Let us focus for a while on the second case. By integration we get $$\dfrac{u_{0,x}}{2}=u_0\tilde m_\pm+\delta_\pm(t),$$ where $\delta_\pm(t)$ are constants of integration which, in general,  must differ from 0. A direct further integration gives $$u_0=\psi_\pm^2(x,t,0)\left(\rho_\pm(t)+\int_{\mp \infty}^x\dfrac{\delta_\pm(t)}{\psi^2_\pm(s,t,0)}ds\right),$$ where $\rho_\pm(t)$ are again constant of integration. By forcing $u_0$ to be bounded on all $\mathbb R$, we simply have $$u_0=\psi_+^2(x,t,0)\int_{-\infty}^x\dfrac{\delta_+(t)}{\psi^2_+(s,t,0)}ds=-\psi_-^2(x,t,0)\int_{x}^{+\infty}\dfrac{\delta_-(t)}{\psi^2_-(s,t,0)}ds,$$ which gives an insight on the decay rate of $\psi_\pm$ as $|x|\rightarrow\pm\infty$ (and of course justifies the fact that $\delta_\pm(t)\neq 0$).
 
 What is perhaps more important is what we call the \it Weyl difference \rm $${\mathcal M_0}(x,t):=\tilde m_-(x,t)-\tilde m_+(x,t).$$
 In the case $(i)$ we have 
 \begin{equation}\label{d1}\eta(0){\mathcal M_0}_t=(u_0{\mathcal M_0})_x,\end{equation}
while in the case $(ii)$ we simply get \begin{equation}\label{d2} (u_0{\mathcal M_0})_x=0.\end{equation}
Again, $u_0$ can be directly obtained from ${\mathcal M_0}$ and vice-versa, and in fact \begin{equation}\label{d3}\begin{cases} u_0(x,t)=\dfrac{1}{{\mathcal M_0}(x,t)}\displaystyle{\left(\delta(t)+\eta(0)\int_0^x{\mathcal M_0}_t(s,t)ds\right)},&\eta(0)\neq 0\\ \\u_0(x,t)=\dfrac{\delta(t)}{{\mathcal M_0}(x,t)},&\eta(0)=0,\end{cases}\end{equation} where, again, $\delta(t)$ is  a constant of integration. 

\bigskip

Another interesting fact concerning the Weyl difference $\mathcal M_0$ is that it defines uniquely the potential $q$. Indeed, let $$\mathcal N_0:=\tilde m_-(x,t)+\tilde m_+(x,t) .$$ The first equation in \eqref{mxt} gives, for $\lambda=0$, $$\begin{cases}\mathcal N_0=-\dfrac{\mathcal M_{0,x}}{\mathcal M_0}\\ \\ \mathcal N_{0,x}+\dfrac{1}{2}\left(\mathcal M_0^2+\mathcal N_0^2\right)=2q,\end{cases}$$ and it follows that 
\begin{equation}\label{mq}-\dfrac{\mathcal M_{0,xx}}{\mathcal M_0}+\dfrac{3}{2}\left(\dfrac{\mathcal M_{0,x}}{\mathcal M_0}\right)^2+\dfrac{\mathcal M_0^2}{2}=2q,\end{equation} which expresses, as anticipated above, $q$ as a function of $\mathcal M_0$. 

A  similar reasoning can be made for the general function $$\mathcal M(x,t,\lambda):=m_-(x,t,\lambda)-m_+(x,t,\lambda),$$ obtaining the following fundamental
\begin{thm}\label{P102} Let $\mathcal M(x,t,\lambda)=m_-(x,t,\lambda)-m_+(x,t,\lambda)$. Then for every $\lambda$ which is not an eigenvalue of \eqref{sl}, we have  \begin{equation}\label{mm}\begin{cases}\eta(\lambda)\mathcal M_t=\left(U\mathcal M\right)_x\\ \\ -\dfrac{\mathcal M_{xx}}{\mathcal M}+\dfrac{3}{2}\left(\dfrac{\mathcal M_{x}}{\mathcal M}\right)^2+\dfrac{\mathcal M^2}{2}=2(q-\lambda y).\end{cases}\end{equation}
\end{thm}
These results can give an idea of at least two alternative different ways  to interpret/define the hierarchies. We discuss in the next line one of these ways, and describe the second one, which is valid for pairs $(q,y)\in\mathcal E_2$ which satisfy an additional assumption, in the next Section.

Using $\mathcal M_0$, whatever the value $\eta(0)$, $u_0$ is directly determined by the evolution of $\mathcal M_0$, and so is $q$. This means that by choosing a function $\mathcal M_0(x,t)$ which is bounded and of class $C^{2r+3}(\mathbb R)\times C^1(\mathbb R)$, we have  well defined $q(x,t)$ and $u_0(x,t)$ which satisfy at once one equation (the last one) in the system \eqref{sys}. So only one equation remains as the compatibility condition for the zero-curvature to hold, namely the equation $H^y_r(k,q)$. As a simple example, by taking $k=r$ and $\mathcal M_0=2$, we have $q(x,t)=1$, and obtain the Camassa-Holm equation of order $r$, namely $H^y_r(r,1)$. However, if we choose to fix $y$, then things become more subtle. We explain this with an example.
\begin{ex1}\rm  Fix $\eta=1$, $r=1$, $k=0$ and $y=1$. Then the corresponding compatibility condition is the well-known K-dV equation. Now, we can always choose to fix a function $\mathcal M_0$ and define $q$ as in \eqref{mq}, in such a way that the last equation in \eqref{sys} is satisfied. Note that this equation was exactly the compatibility condition which we considered before. Where did the compatibility condition go? Well, it moved  in the second equation of \eqref{sys} which states that $u_0$ must be equal to $q/2$, and hence that the quantity in the right-hand  side of the first equation in \eqref{d3} equals $q/2$. We thus obtain an evolution equation for $\mathcal M_0$, namely \begin{equation}\label{mK}\mathcal M_{0,t}=\dfrac{1}{4}\left(-\mathcal M_{0,xx}+\dfrac{3}{2}\dfrac{(\mathcal M_{0,x})^2}{\mathcal M_0}+\dfrac{\mathcal M_0^3}{2}\right)_x.\end{equation}

The equation \eqref{mK} gives the evolution equation for $\mathcal M_0$ in the case that $q$ satisfies the K-dV equation, and vice-versa, for  a given (sufficiently regular) solution of the equation \eqref{mK}, the function $q(x,t)$ defined via \eqref{mq} satisfies the K-dV equation.  The equation  \eqref{mK} is equivalent to the K-dV equation, although it appears to be more difficult to study. However, the simplicity of the procedure in the case when $q=1$ is only apparent, because there is a quite involved relation which is to be used to determine the function $y(x,t)$ (see \cite{JZ1,JZ2,JZ5,JZ6}).
\end{ex1}

\section{The hierarchy for certain classes of pairs $(q,y)\in\mathcal E_2$}

We discuss a method to define the hierarchy by means of the expansion of the Weyl functions in a neighborhood of the point $z=\infty$ on the Riemann sphere.
First, we state some properties of the Weyl $m$-functions. To simplify the notation, let us first consider the functions $m_\pm(\lambda):=m_\pm(0,\lambda)$. These functions are holomorphic in the union $\mathbb C^+\cup\mathbb C^-$, and moreover $$sgn\;\dfrac{\Im m_\pm(\lambda)}{\Im\lambda}=\pm 1.$$ It follows easily that both $m_+(\lambda)$ and $-m_-(\lambda)$ are Herglotz functions. For information concerning Herglotz functions, in particular in relation to the Weyl $m$-functions, see \cite{Kot}. It is however important to notice that both $m_\pm(\lambda)$ admits non-tangential limits $$\lim_{\ep\rightarrow0}m_\pm(\lambda+i\ep)=:m_\pm(\lambda+i0),$$ for a.a. $\lambda\in\mathbb R$. We address the reader to \cite{Kot} for more information. Now, we can define a function in the union $\mathbb C^+\cup\mathbb C^-$ as follows: $$h(\lambda)=\begin{cases}m_+(\lambda),&\Im\lambda>0\\m_-(\lambda),&\Im\lambda<0.\end{cases}$$ This function is holomorphic in $\mathbb C^+\cup\mathbb C^-$,  and in fact in $\mathbb C\setminus \Sigma_a$. By using the Bebutov flow, we can define functions $m_\pm(x,\lambda)$, having the same properties as above, for $x\in\mathbb R$. The assumption we make here is about the possibility of expanding $h(x,\lambda)$ in a real half-line.
To do this, we need a fundamental assumption.
\begin{hy1}\label{H2} Consider the operator $\mathcal L_a=\dfrac{1}{y}[-D^2+q]$, acting on $L^2(\mathbb R,ydx)$. Define $$h(x,\lambda)=\begin{cases}m_+(x,\lambda),&\Im\lambda>0\\m_-(x,\lambda),&\Im\lambda<0.\end{cases}$$ Let $z^2=-\lambda$ be a local parameter near $\infty$ in the Riemann sphere $\mathcal R$. We assume that there exists $N\in\mathbb N\cup\{\infty\}$ such that, in a neighborhood of $z=\infty$, the function $h(x,z)$ admits an expansion \begin{equation}\label{EE} h(x,z)=a_1(x)z+a_0(x)+\sum_{j=1}^{N}a_{-j}(x)z^{-j}+\mathcal O(z^{-N-1}).\end{equation}
\end{hy1}

\begin{defin} The set $\mathcal P_N$ denotes the subset of $\mathcal E_2$ consisting of the pairs $a=(q,y)\in\mathcal E_2$ satisfying Hypotheses \ref{H1} and  \ref{H2}.\end{defin}

We will provide in the following some families of pairs $(q,y)$ which satisfy Hypothesis \ref{H2}.  

It is important to notice that  Hypothesis \ref{H2} is a real restriction on the admissible pairs in $\mathcal E_2$, and in fact it implies immediately that \begin{equation}\label{expn1}\begin{cases}m_+(x,z)=a_1(x)z+a_0(x)+\displaystyle{\sum_{j=1}^N}a_{-j}(x)z^{-j}+\mathcal O(z^{-N-1})\\ \\
m_-(x,z)=-a_1(x)z+a_0(x)+\displaystyle{\sum_{j=1}^N}(-1)^ja_{-j}(x)z^{-j}+\mathcal O(z^{-N-1}),\end{cases}\end{equation} i.e., a precise indication concerning the behavior of the Weyl $m$-functions. 

\bigskip

 Introducing a time dependence, and assuming that the zero-curvature relation \eqref{zc1} is valid, we can obtain the same expansions, where the coefficients depend on $x$ and $t$. Using the Riccati equation for $m_\pm$ we have \begin{equation}\label{coeff}\begin{cases}a_1=-\sqrt y\\ \\a_0=\dfrac{-y_x}{4y}\\ \\ a_{0,x}+a_0^2+2a_1a_{-1}=q \\ \\  a_{-j}=\dfrac{-1}{2a_1}\left[a_{-j+1,x}+\displaystyle{\sum_{k=-j+1}^0}a_ka_{-j+1-k}\right],& 2\leq j\leq N\end{cases}\end{equation} Note that $a_1$ and $a_{-1}$ determine \it all \rm the other coefficients. One can think at the above relations as $a_1$ which defines $y$ and $a_{-1}$ which defines $q$ via the relation \begin{equation}\label{defq}q=a_{0,x}+a_0^2-2\sqrt ya_{-1}=\tilde q+2\sqrt ya_{-1},\end{equation} where $\tilde q$ is as in \eqref{tq}. All the other coefficients are determined via the remaining relations in \eqref{coeff}. Now, consider the function $\mathcal M(x,t),$ as above. Then the first equation in \eqref{mm} is valid, i.e., $$\eta(\lambda)\mathcal M_t=(U\mathcal M)_x.$$ The expansion of $\mathcal M$ is given by \begin{equation}\label{exp3}\mathcal M=-2a_1z-2\sum_{j=0}^{N_1} a_{-(2j+1)}z^{-(2j+1)}+O(z^{-N_1}),\end{equation} where $N_1=\displaystyle{\left\lfloor{\dfrac{N-1}{2}}\right\rfloor},$ and $\left\lfloor{\cdot}\right\rfloor$ denotes the integer part. Now, let us choose arbitrarily two coefficients $a_1$ and $a_{-1}$ of the expansion for $\mathcal M$. Note that, as argued above, $a_1$ defines a function $y$, while $a_{-1}$ defines $q$. Assume that Hypotheses \ref{H} is valid. All the coefficients of the expansion of $\mathcal M$ are determined by means of \eqref{coeff}. 

We try to determine the polynomial $U$ such that the evolution $\mathcal M_t=(U\mathcal M)_x$ holds. Inserting the expression for $U$ and the expansion of $\mathcal M$ up to the polynomial part in the evolution equation we obtain \begin{equation}\label{exp2} \begin{split}&(-1)^{k}z^{2k}\left(a_{1,t}z+\sum_{j=0}^{N_1} a_{-(2j+1),t}\cdot z^{-(2j+1)}\right)=\\&\left[\left(\sum_{i=0}^r(-1)^iu_{i}z^{2i}\right)\left(a_{1}z+\sum_{j=0}^{N_1} a_{-(2j+1)}\cdot z^{(-2j+1)}\right)\right]_x\end{split}\end{equation} The expression \eqref{exp2} provides a certain number of relations, ordered by the decreasing power of $z$. In order to completely determine the coefficients of the polynomial $U$ we need exactly $r+1$ relations, while another condition is necessary in order for all the other relation to be satisfied. Such an additional relation is exactly the compatibility condition we are looking for, and is in fact an equivalent expression of the evolution equation we are interested in. This reasoning implies that, once we have fixed a number $r\geq 1$, then at least we need that $$N_1=\displaystyle{\left\lfloor{\dfrac{N-1}{2}}\right\rfloor}\geq r, $$ or, equivalently $$N\geq 2r+1.$$ All the conditions can be expressed in terms of $a_1$, $a_{-1}$ and their derivatives, and hence in terms of $q$, $y$ and their derivatives. We address the reader to  future publications \cite{Z4,Z5} for the details on the general case, and here, mostly for notational convenience, we only furnish the example of the standard K-dV equation.
For, we have to consider the eigenvalue equation $-\varphi_{xx}+q\varphi=\lambda \varphi,$ so that   $a_1=-1$, $a_0=0$ and $a_{-1}=-\dfrac{q}{2}.$ This fact is equivalent to choose $a_1=-1$ and $a_{-1}=-\dfrac{q}{2}.$ We take $k=0$ and  $U=u_1\lambda+u_0$, so that $N_1\geq 1$, or $N\geq 3$ (that is, we need at least three terms in the polynomial expansions \eqref{expn1}). We write explicitly the relations in \eqref{exp2}:
\begin{equation}\label{REC}\begin{cases} -[u_1a_1]_x=0\\ \\ [-u_1a_{-1}+u_0a_{1}]_x=0\\ \\  a_{-1,t}=[-u_1a_{-3}+u_0a_{-1}]_x\\ a_{-(2j+1),t}=[-u_1a_{-(2j+3)}+u_0a_{-(2j+1)}]_x,&j\geq 1\end{cases}.\end{equation}
The first equation determines $u_1=1$, the second one gives $u_0=\dfrac{q}{2}$ and the third one is the compatibility condition for which all the other equations in the system are satisfied, namely we have $$-\dfrac{q_t}{2}=\left[-\dfrac{q^2}{4}-a_{-3}\right]_x.$$ But using \eqref{coeff} we have $$a_{-3}=\dfrac{1}{8}\left(q^2-q_{xx}\right)$$ and a direct substitution gives $$\dfrac{q_t}{2}=\left[\dfrac{q^2}{4}+\dfrac{1}{8}(q^2-q_{xx})\right]_x=-\dfrac{q_{xxx}}{8}+\dfrac{3}{4}qq_x,$$ that is, the K-dV equation! By construction, \it all \rm the other relations are satisfied. We test directly the first of them: $$a_{-3,t}=\left[-a_{-5}-\dfrac{q}{2}a_{-3}\right]_x.$$ We use \eqref{coeff} to obtain $$a_{-5}=\dfrac{1}{16}\left[\dfrac{5}{2}(q_x)^2+3qq_{xx}-\dfrac{1}{2}q_{xxxx}-q^3\right],$$
$$a_{-3,t}=\dfrac{1}{8}\left[\left(-\dfrac{q_{xxx}}{4}+\dfrac{3}{2}qq_x\right)_{xx}-2q\left(-\dfrac{q_{xxx}}{4}+\dfrac{3}{2}qq_x\right)\right].$$ Again, a direct substitution gives the desired relation. The others can be proved by induction, as stated in the  following
\begin{thm}\label{T101} Let us consider the case of the K-dV equation, i.e., the case when $y(x,t)\equiv 1$, $\eta(\lambda)\equiv 1$ and $r=1$. Assume that Hypotheses \ref{H1}  and \ref{H2} hold. Then for every $j=1,\dots,N$ we have $$a_{-j,t}=[-a_{-j-2}-a_{-1}a_{-j}]_x.$$\end{thm}
\begin{proof} First of all, notice that, in this case, we have $u_1=1$ and $u_0=-a_{-1}$. By induction,
assume that all the relations in \eqref{REC}  are satisfied up to a certain order $j$, for every coefficient of the expansion \eqref{expn1}, i.e.,  $$\dot{a_{-j}}=[-a_{-j-2}-a_{-1}a_{-j}]',\;\;j\geq 1$$ (we are using a more convenient notation for the derivatives with respect to $t$ and $x$).  We can directly check the relations for $j=1$, $j=2$ and $j=3$, as above. Next, we compute the time-derivative of the coefficient $a_{-j-1}$ by using \eqref{coeff}, as follows:  $$\dot{a_{-j-1}}=\dfrac{1}{2}\left[\dot{a_{-j}}'+\sum_{k=-j}^0\dot{a_{k}}a_{-j-k}+\sum_{k=-j}^0a_k\dot{a_{-j-k}}\right]=$$ $$=\dfrac{1}{2}\left[-a_{-j-2}''-(a_{-1}a_{-j})''+\sum_{k=-j}^0a_{-j-k}(-a_{k-2})'+\sum_{k=-j}^0a_k(-a_{-j-k-2})'\right]+$$ $$\dfrac{1}{2}\left[\sum_{k=-j}^0-a_{-j-k}(a_{-1}a_k)'+\sum_{k=-j}^0-a_k(a_{-1}a_{-j-k})'+a_{-1}''a_{-j}-(a_{-1}a_{-j})''\right].$$ Now, since $$2a_{j-3}'=[a_{-j-2}'+\sum_{k=-j-2}^0a_ka_{-j-2-k}]',$$ one can see that $$\dot{a_{-j-1}}=-a_{-j-3}'+\dfrac{1}{2}\Big[-2a_{-1}'a_{-j-1}+a_{-1}''a_{-j}+$$ $$+\sum_{k=-j}^0a_{-j-k}(a_{-1}a_k)'+\sum_{k=-j}^0a_k(a_{-1}a_{-j-k})'-(a_{-1}a_{-j})''\Big].$$ It follows, looking at the relations in \eqref{exp2}, that the proof follows if
$-(a_{-1}a_{-j-1})'$ equals the part into the square parenthesis in the above formula. Computing directly such a part, we have $$\dfrac{1}{2}\Big[-a_{-1}'(2a_{-j-1}+2(-2a_{-j-1}-a_{-j}'))-a_{-1}(-a_{-j}''-2a_{-j-1}')-$$ $$-(a_{-1}a_{-j})''+a_{-1}''a_{-j}\Big]=-(a_{-1}a_{-j-1})',$$ as desired. 
\end{proof}

We finish this Section by underlying an important consequence of the reasoning made up to this point. One can choose \it a priori \rm an asymptotic expansion of the form \eqref{exp3} of a function $\mathcal M(x,t,z)$ near $z=\infty$ in the Riemann surface of $w^2=-\lambda$ by assigning $a_1$ and $a_{-1}$, then defining all the other coefficients as in \eqref{coeff}. The coefficients $a_1$ and $a_{-1}$ in turn define the density $y(\cdot,t)$ and the potential $q(\cdot,t)$ respectively of a Sturm-Liouville operator $$\dfrac{1}{y(\cdot,t)}\left[-D^2+q(\cdot,t)\right].$$ We have the following Theorem, giving the situation in the general case

\begin{thm}\label{T102} The validity of the zero-curvature relation, under the conditions expressed by Hypotheses \ref{H1} and \ref{H2}, is  equivalent to a compatibility condition between $a_1$, $a_{-1}$ and the coefficients of the polynomial $U$, and is given by equating the coefficients of the power $z^{-1}$ in \eqref{exp2}, i.e., \begin{equation}\label{fin1} (-1)^k\cdot a_{-(2k+1),t}=\left[\sum_{j=0}^r(-1)^ju_ja_{-(2j+1)}\right]_x,\end{equation} where the coefficients $a_j$ are given as in \eqref{coeff}, i.e.,  $$a_1=-\sqrt y, \;a_{-1}=\dfrac{\tilde q-q}{2\sqrt y},\; \tilde q=-\left(\dfrac{y_x}{4y}\right)_x+\left(\dfrac{y_x}{4y}\right)^2,$$ $$a_{-j}=\dfrac{-1}{2a_1}\left[a_{-j+1,x}+\displaystyle{\sum_{k=-j+1}^0}a_ka_{-j+1-k}\right],\;\;2\leq j\leq N,$$ under the condition that $N\geq 2r+1$, or, equivalently, $N_1:=\displaystyle{\left\lfloor\dfrac{N-1}{2}\right\rfloor\geq r}.$
\end{thm}

 The condition \eqref{fin1} is an evolution equation which relates the functions $q$ and $y$. This means that the hierarchies can be completely defined by introducing a power series at $z=\infty$ in which the leading coefficients $a_1$ and $a_{-1}$ satisfy a certain compatibility condition, expressed by \eqref{fin1}, and all the other coefficients of the expansion are determined by \eqref{coeff}. Note that, however, the expansion \eqref{exp3} need not  be valid away from a disc centered at  $z=\infty$ in $\mathcal R$, and  that the expansion can be performed for any arbitrary $N$. 

Another important observation is that, while the relations in \eqref{REC} hold for any choice of $j$, up to the value of  of $N$ for which the expansions \eqref{expn1} hold, as stated in Theorem \ref{T102}. In the case when \eqref{expn1} translates into a full Laurent expansion (i.e., $N=\infty$), then there is an infinite number of relations which are satisfied once \eqref{fin1} holds: below, we will provide same explicit cases in which this situation occurs.
\section{Description of some subsets of  $\mathcal P_N$}
In this section we will provide examples of pairs $a=(q,y)\in\mathcal E_2$ which belong to $\mathcal P_N$, for some $N\in \mathbb N\cup\{\infty\}$.

\bigskip

\begin{center}{ \bf Algebro-Geometric Potentials}\end{center}
We briefly describe the   the so-called algebro-geometric potentials. For more information, see \cite{DMN,FJZ,JZ1,JZ2}. In this case, the spectrum of the operator $\mathcal L_a$ is a set $$\Sigma_a=[\lambda_0,\lambda_1]\cup[\lambda_2,\lambda_3]\cup\ldots\cup[\lambda_{2g},\infty)$$ and moreover $m_+(\lambda+i0)=\overline{m_-(\lambda+i0)}$ in $(\lambda_{2g},\infty)$. This fact implies that the expansion in \eqref{expn1} is  full Laurent expansion, i.e., $N=\infty$, and this means that algebro-geometric potentials lie in $\mathcal P_\infty$. In particular, such potentials are of class $C^\infty(\mathbb R)$. 

Actually, for algebro-geometric potentials,  it can be proved much more (see, again, \cite{FJZ,JZ1}): the Weyl $m$-functions $m_\pm(x,\lambda)$ define a single meromorphic function $M(x,z)$ on the Riemann surface $\mathcal R_g$  of the algebraic relation $w^2=-(\lambda-\lambda_0)(\lambda-\lambda_1)\cdot\ldots\cdot(\lambda-\lambda_{2g}).$ Such a function $M(x,z)$ has a pole in each of the so-called \it spectral gaps \rm $I_1=[\lambda_1,\lambda_2], I_g=\ldots, [\lambda_{2g-1},\lambda_{2g}]$, and in fact if $P_j(x)\in I_j$ ($j=1,\ldots,g$), then either $|m_+(P_j(x)+i\ep)|$ or $|m_-(P_j(x)+i\ep)|$ is singular as $\ep\rightarrow 0$, and this singularity is a simple pole. Note that, in general, although the full-line spectrum is invariant under the translation flow, the eigenvalues of the half-line operators are not, and indeed they rotate, as $x$ varies, in circles $c_j=\pi^{-1}(I_j)$, where, $\pi$ is the standard projection $\pi:\mathcal R_g\rightarrow \mathbb C\cup\{\infty\}$. Well, in this case, the Weyl $m$-functions are fully determined by formulas involving the Weyl difference $\mathcal M_0(x)$ and the poles $P_j(x)$, as follows.  
 If $P\in\mathcal R_g$, let $$k(P)=\sqrt{-(P-\lambda_0)(P-\lambda_1)\cdot\ldots\cdot(P-\lambda_{2g})}.$$
First of all, the motion of the points $P_j(x)$ (the so-called \it pole motion\rm)  is fully determined by the relations 
$$P_{j,x}(x)=\dfrac{(-1)^g\mathcal M_0(x)k(P_j)\displaystyle{\prod_{i=1}^g P_i(x)}}{k(0)\displaystyle{\prod_{s\neq j}(P_j(x)-P_s(x))}},$$ where $k(0)$ is the positive square root of $\lambda_0\lambda_1\cdot\ldots\cdot\lambda_{2g}.$
If $\Im\lambda\neq 0$, let $P_\lambda\in \pi^{-1}(\lambda)$ be such that $\Im\lambda\Im P_\lambda>0$.  Finally, let $$H(x,\lambda)=\dfrac{2(-1)^{g+1}k(0)}{\mathcal M_0(x)\displaystyle{\prod_{i=1}^g P_i(x)}}\cdot\prod_{i=1}^g(\lambda-P_i(x)).$$ Then $$m_\pm(x,\lambda)=\dfrac{\pm 2k(P_\lambda)+H_x(x,\lambda)}{2H(x,\lambda)}.$$ These relations testify that the choice of a bounded function $\mathcal M_0(x)$ and the initial $P_1(0),\ldots,P_g(0)$ completely determine the moving poles $P_j(x)$ and the Weyl $m$-functions. Moreover, one has the so-called \it trace formulas \rm $$y(x)=\dfrac{\mathcal M_0^2(x)\displaystyle{\prod_{i=1}^g P_i^2(x)}}{4k^2(0)},$$ and $$q(x)=y(x)\left(\lambda_0+\sum_{i=1}^g[\lambda_{2i}+\lambda_{2i-1}-2P_i(x)]\right)+\tilde q(x),$$ where $\tilde q(x)$ is as in \eqref{tq}.

All these facts show that, in the algebro-geometric case, the choice of $\mathcal M_0(x)$ and $P_1(0),\ldots,P_g(0)$ determine both $q(x)$ and $y(x)$ as well. As an example, if $\mathcal M_0(x)=2$, then one has $q(x)=1$ and $$y(x)=\dfrac{1}{k^2(0)}{\displaystyle{\prod_{i=1}^g P_i^2(x)}}.$$ If, instead, we want to retrieve the Schr\"odinger operator, then we have to obtain $y(x)=1$, hence we must choose the pole motion in such a way that 
$$P_{j,x}(x)=\dfrac{-2k(P_j(x))}{\displaystyle{\prod_{s\neq j}(P_j(x)-P_s(x))}},$$ so that 
$$1=\dfrac{\mathcal M_0^2(x)\displaystyle{\prod_{i=1}^g P_i^2(x)}}{4k^2(0)}.$$ See \cite{FJZ,JZ1} for a detailed discussion of these matters. 

It appears now to be clear that the introduction of a time dependence of both $\mathcal M_0$ and the poles $P_j$ will determine the hierarchy and its solutions. It can be proved (see, for instance \cite{JZ1,JZ5,JZ6}) that this can be done by the choice of a bounded function $\mathcal M_0(x,t)$ and the requirement that the poles $P_j(x,t)$ satisfy $$\begin{cases}P_{j,x}(x,t)=\dfrac{(-1)^g\mathcal M_0(x,t)k(P_j(x,t))\displaystyle{\prod_{i=1}^g P_i(x,t)}}{k(0)\displaystyle{\prod_{s\neq j}(P_j(x,t)-P_s(x,t))}}\\ \\ P_{j,t}(x,t)=\dfrac{U(x,t,P_j(x,t))}{P_j^k(x,t)}\cdot P_{j,x}(x,t),\end{cases}$$ where $U(x,t,\lambda)$ is a polynomial of degree $r$ in $\lambda$, whose coefficients are determined recursively
from the relation
\begin{equation*}\begin{split}U_{x}(\lambda)&=\lambda^k\left(\dfrac{\mathcal M_{0,t}}{\mathcal M}\right)-\dfrac{\mathcal M_{0,x}}{\mathcal M}U(\lambda)+\sum_{i=1}^g\left[\dfrac{\lambda^k}{P_i^k} U(P_i)- U(\lambda)\right]\dfrac{\lambda P_{i,x}}{P_i(\lambda-P_i)}.\end{split}
\end{equation*}
It turns out that the coefficients of the polynomial $U$ defined above satisfy \eqref{sys}, hence the zero-curvature relation \eqref{zc1} and the functions $q(x,t)$ and $y(x,t)$ defined via the trace formulas satisfy evolution equations which make consistent all the structure of the zero-curvature.  We send the reader back to \cite{Z1} for a detailed discussion.  A well-known example of this kind of construction is that of the K-dV hierarchy with algebro-geometric initial data in \cite{DMN}. Another is given for the Camassa-Holm hierarchy with algebro-geometric initial data (see \cite{GH,Z1}), and for the general case of the  Sturm-Liouville hierarchy in \cite{JZ5,JZ6,JZ7}. The above cited papers contain also a discussion concerning the expressions of the the solutions of the $r$-th order equation of the hierarchy, by transferring them to an appropriate generalized Jacobi variety related to the spectral problem.

\bigskip

\begin{center}{\bf Decaying Potentials}\end{center}

In analogy with the theory of decaying potentials for the Schr\"odinger operator, we now describe some aspects of the extension to the more general case of the Sturm-Liouville operator: more information and a detailed discussion can be found in \cite{Z3}. Let us consider a pair $a=(q,y)\in\mathcal E_2$ satisfying the following assumption:
\begin{hy1} There exists a constant $\lambda_0>0$ such that $$V(x):=\dfrac{q(x)-\tilde q(x)}{y(x)}-\lambda_0\in L^1(\mathbb R).$$\end{hy1}
The assumption that $\lambda_0>0$ can be relaxed with the choice of any real number, and the reasoning we will make will be equivalent, by using a simple translation argument. Under this assumption, it can be proved that the absolutely continuous spectrum of the operator $\mathcal L_a$ is the set $[\lambda_0,\infty)$. Setting $k^2=\lambda-\lambda_0$, for $\Im k\geq 0$ it is well-known that there exist the \it Jost \rm solutions $\varphi_\pm(x,k)$ of the eigenvalue equation $$-\varphi''(x)+q(x)\varphi(x)=k^2y(x)\varphi(x),$$ satisfying $$\varphi_\pm(x,k)\sim y^{-1/4}(x)e^{\pm ik\mathcal I(x)},$$ as $x\rightarrow\pm\infty,$ where $$\mathcal I(x)=\int_0^x\sqrt{y(s)}ds.$$
We have $$m_\pm(x,\lambda)=\dfrac{\varphi_\pm'(x,k)}{\varphi_\pm(x,k)},$$ and one can also prove that $m_\pm$ can be continuously extended to $\lambda=0$. There are asymptotic expansions for $\varphi_\pm(x,k)$, which hold as soon as $V(x)$ defined above is of class $C^N(\mathbb R)$, namely $$y^{1/4}(x)e^{\mp ik\mathcal I(x)}\varphi_\pm(x,k)=1+\sum_{j=0}^{N}(-1)^{j}\alpha_j(x)(2ik)^{-j}+o(k^{-N}).$$ In particular, the coefficients $\alpha_j(x)$ can be determined by $$\begin{cases}\alpha_1(x)=\displaystyle{-\int_x^\infty V(s)ds}\\ \\ \alpha_{j+1}(x)=\alpha_j'(x)-\displaystyle{\int_x^\infty V(s)\alpha_j(s)ds}.\end{cases}$$ Differentiating the above relations, we obtain the expansions for the Weyl $m$-functions $$m_+(x,z)=-\sqrt{y(x)}z-\dfrac{y'(x)}{4y(x)}+\sum_{j=1}^N\beta_{-j}(x)z^{-j}+o(-N)$$ and $$m_-(x,z)=\sqrt{y(x)}z-\dfrac{y'(x)}{4y(x)}+\sum_{j=1}^N(-1)^{j}\beta_{-j}(x)z^{-j}+o(-N).$$ These relations imply immediately that  Hypothesis \ref{H2} is satisfied, hence the pair $a=(q,y)\in \mathcal P_N$.

\bigskip

\begin{center}{\bf Scattering Potentials}\end{center}

 These potentials arise as a particular case of the above class of decaying potentials, and can be defined as those pairs $(q,y)\in\mathcal E_2$ such that $$\int_\mathbb R(1+|\mathcal I(x)|)\cdot |V(x)|dx<\infty,$$ where the functions $V(x)$ and $\mathcal I(x)$ have been defined above.
It follows immediately from the above discussion that this potentials lie in $\mathcal P_N$ as soon as $V\in C^N(\mathbb R)$. We can obtain, however and important object from the discussion concerning these potentials, namely, the (classical) reflection coefficient. Considering the equation $$-\varphi'(x)+q(x)\varphi(x)=k^2y(x)\varphi(x)$$ as above, with $k^2=\lambda-\lambda_0$, we can see that there exist solutions $\psi_\pm(x,k)$ of such equation, satisfying $$\begin{cases}\psi_+(x,k)\sim y^{-1/4}(x)e^{ik\mathcal I(x)},&x\rightarrow+\infty\\ \psi_-(x,-k)\sim y^{-1/4}(x)e^{-ik\mathcal I(x)},&x\rightarrow-\infty.\end{cases}$$ Since the pairs $(\psi_+(x,k),\psi_+(x,-k))$ and $(\psi_-(x,k),\psi_-(x,-k))$ are independent  for nonzero real values of $k$, there must exist constants $a(k)$ and $b(k)$ such that $$\begin{cases}\psi_+(x,k)=a(k)\psi_-(x,k)+b(k)\psi_-(x,-k)\\\psi_-(x,-k)=a(k)\psi_+(x,-k)+b(-k)\psi_+(x,k).\end{cases}$$ Since $V(x)$ is real, for nonzero real values of $k$, we must have $$\begin{cases} a(k)=\overline{a(-k)}\\b(k)=\overline{b(-k)}\\|a(k)|^2=1+|b(k)|^2.\end{cases}$$ We set $$r^+(k):=-\dfrac{b(-k)}{a(k)},\;\;\;r^-(k):=\dfrac{b(k)}{a(k)}.$$ These quantities are called the {\bf  right  reflection coefficient}  and the {\bf left reflection coefficient} respectively. By using these quantities, together with some of the spectral data of the operator, one can solve the inverse spectral problem. This fact is very well-known for the Schr\"odinger operator (and dates back to the pioneering works of Gel'fand, Levitan, Marchenko and their school \cite{Le,Le1,Ma,Ma1,Mar}). More recently, a method to obtain information to solve the direct and inverse spectral problem for the Schr\"odinger operator by means of a approximation procedure has been developed in \cite{krav}. In the case of the general Sturm-Liouville operator, one can see the recent papers \cite{Con1,Con2,Con3,Z3}. Anyway, from the definition of $m_\pm$, one can see that $$m_+(x,k^2)=\dfrac{\psi_+'(x,k)}{\psi_+(x,k)},\;\;\;m_-(x,k^2)=\dfrac{\psi_-'(x,-k)}{\psi_-(x,-k)},$$ and since the Weyl $m$-functions admit the non tangential limits $m_\pm(\lambda+i0)$ for real values of $\lambda$, we have, for real values of $k$ (hence $\lambda=k^2>0$), \begin{equation}\label{refl1}m_+(x,k^2+i0)-\overline{m_-(k^2+i0)}=\dfrac{2ik b(k)}{\psi_+(x,k)\overline{\psi_-(x,-k)}}.\end{equation} The relation \eqref{refl1} show a direct correspondence between the reflection coefficient and the Weyl $m$-functions. 
\begin{defin} A scattering potential satisfying $b(k)=0$ (hence $r^\pm(k)=0$) is called {\bf reflectionless}. \end{defin} From \eqref{refl1} it is immediately seen that a potential is reflectionless if and only if $m_+(x,k^2+i0)-\overline{m_-(k^2+i0)}=0$. For more information concerning reflectionless potentials and their generalizations, see \cite{JZ3,JZ4,Lu,Ko}. The behavior of the Weyl $m$-functions that we have obtained allows us to apply a generalized version of the Schwarz reflection principle \cite{Du,NS}, to claim that Hypothesis \ref{H2} holds {\it for every} $N>0$, hence that the expansion \eqref{EE} is a full Laurent expansion at $z=\infty$. 

\bigskip

\begin{center}{\bf The generalized reflection coefficient and the corresponding potentials}\end{center}

We extend the concept of the reflection coefficient to a more general class of potentials, with the aim of finding a condition in order Hypothesis \ref{H2} to hold. 
Let us consider the functions $$\mathcal G(x,\lambda)=\dfrac{y(x)}{\mathcal M(x,\lambda)}=\dfrac{y(x)}{m_-(x,\lambda)-m_+(x,\lambda)},\;\;\mathcal H(x,\lambda)=m_-(x,\lambda)m_+(x,\lambda)\mathcal G(x,\lambda)$$ and $$R(x,\lambda)=\dfrac{\overline{m_+(x,\lambda)}-m_-(x,\lambda)}{m_+(x,\lambda)-m_-(x,\lambda)}.$$ It follows easily that both $\mathcal G$ and $\mathcal H$ are Herglotz functions (i.e. they map $\mathbb C^+$ into $\mathbb C^+$) and that $|R(x,\lambda)|\leq 1$.  Now, the spectrum of the operator $\mathcal L_a$ is bounded below by some constant $\lambda_\ast$, and both $m_+(\lambda+i\ep)$ and $m_-(\lambda+i\ep)$ are real valued for $\lambda<\lambda_\ast$. Since $\ln\mathcal G$ and $\ln\mathcal H$ are still Herglotz function, by using well-known integral representations of Herglotz functions, it can be proved (see \cite{Kot}) that both $\ln \mathcal G$ and $\ln\mathcal H$  have Laurent  expansions of order $N>0$ at $\infty$, whenever we have \begin{equation}\label{R}\int_{\lambda_\ast}^\infty\lambda^N|R(x,\lambda+i0)|d\lambda<\infty\end{equation} for $x\in\mathbb R$, and  hence also $\mathcal G$ and $\mathcal H$ will have such expansions. Assuming \eqref{R}, we will have, for $z^2=-\lambda,$ $$\mathcal G(x,z)=\dfrac{y(x)}{m_-(x,z)-m_+(x,z)}=\dfrac{1}{2z}\left(b_1(x)+\sum_{k=0}^Nb_k(x)z^{-2k}\right)+O(z^{-2N-1})$$ and $$\mathcal H(x,z)=\dfrac{z}{2}\left(c_1(x)+\sum_{k=0}^Nc_k(x)z^{-2k}\right)+O(z^{-2N-1}).$$ By using the relations defining $\mathcal G$ and $\mathcal H$ as functions of $m_\pm$, we can easily prove that $m_+(x,z)$ and $m_-(x,z)$  admit expansions as in \eqref{expn1}, hence $h(x,z)$ admits the expansion  \eqref{EE}. We have proved the following \begin{thm}Let Hypothesis \ref{H1} be satisfied. Assume that \eqref{R} holds, i.e., $$\int_{\lambda_\ast}^\infty\lambda^N|R(\lambda+i0)|d\lambda<\infty$$ for $\lambda_\ast\leq\inf\Sigma_a$ and for some number $N\geq 3$. Then Hypothesis \ref{H2} holds, and the hierarchy can be defined, whenever $N\geq 2r+1$.\end{thm}

It is immediate to see that if $a=(q,y)\in\mathcal E_2$ is reflectionless, then \eqref{R} is satisfied, since $R(x,\lambda+i0)=0$, and the expansion is a full Laurent expansion.

\section*{Concluding remarks}
\begin{enumerate}\item In this paper, we introduced the Sturm–Liouville hierarchy of evolution equations by first employing the zero-curvature method. In this framework, the evolution equations appear as compatibility conditions required for the zero-curvature relation to hold. We then translated these compatibility conditions into the language of Weyl $m$-functions, which are fundamental tools in the spectral theory of the associated Sturm–Liouville operator. In particular, we showed that the hierarchy of evolution equations is equivalent to an evolution equation for the Weyl difference $\mathcal M = m_- - m_+$, namely $$\eta(\lambda)\mathcal M_t=(U\mathcal M)_x.$$ Under the assumption that $\mathcal M$ admits a Laurent expansion order $N$ in a neighborhood of $z=\infty$ in the Riemann sphere, we have been able to determine recursion formulas for the coefficients of the expansions, each of them being an evolution equation which is equivalent to the leading evolution equation obtained by means of the zero-curvature relation.
\item The evolution equation for $\mathcal M$ will play a central role in constructing solutions to the hierarchy. This will be achieved by analyzing the evolution of certain spectral parameters of the associated Sturm–Liouville operator under the flow defined by the hierarchy. This program will be carried out in the following two papers of the series \cite{Z4,Z5}, where we will examine in detail both the direct and inverse spectral theory of the Sturm–Liouville operators in question, as well as the time evolution of the spectral parameters involved in reconstructing the pairs $a = (q,y)$ appearing in the evolution equations.
\item Along with the equation $\eta(\lambda),\mathcal M_t = (U\mathcal M)_x$, the condition $a = (q,y) \in \mathcal P_N$ provides a method for determining the time evolution of one of the spectral parameters. Meanwhile, the local behavior of $\mathcal M$ near its poles—correspon\-ding to eigenvalues of the half-line restricted operators—yields the time evolution of the other spectral parameters required to solve the hierarchy. We will also devote particular attention to the choice of initial data for the evolution equations. To this end, we will employ a suitable form of the Liouville transform, which allows us to study both the spectral problems and the associated Weyl $m$-functions in a unified way. In certain cases, working directly with Weyl $m$-functions will make it possible to circumvent the implicit relations that arise in expressing the pair $a = (q,y)$ in terms of the spectral data.

\end{enumerate}

\section*{Acknowledgement and funding} This study was partly funded by the Unione europea-Next Generation EU, Missione 4 Componente C2- CUP Master: J53D2300390 0006, CUP: J53D23003920 006- Research project of MUR (Italian Ministry of University and Research) PRIN 2022 “Nonlinear differential problems with applications to real phenomena” (Grant Number: 2022ZXZTN2).

The authors are members of the “Gruppo Nazionale per l’Analisi Matematica, la Probabilità e le loro Applicazioni” (GNAMPA) of the Istituto Nazionale di Alta Matematica (INdAM) and of GRUPPO DI LAVORO UMI- Teoria dell’Approssima\-zione e Applicazioni- T.A.A. The  authors are  members of “Centro di Ricerca Interdipartimentale Lamberto Cesari” of the University of Perugia.

\end{document}